\documentclass{amsart}

\usepackage{amsmath}
\usepackage{amssymb}
\usepackage{amsthm}
\usepackage{mathscinet}
\usepackage{graphicx}
\usepackage{subcaption}
\usepackage{enumitem}

\theoremstyle{plain}
\newtheorem{thm}{Theorem}[section]

\newtheorem{lem}[thm]{Lemma}

\newtheorem{propalph}{Proposition}
\newtheorem{lemalph}[propalph]{Lemma}

\theoremstyle{definition}
\newtheorem{defn}[thm]{Definition}
\newtheorem*{defn*}{Definition}
\newtheorem{claim}{Claim}[thm]
\newtheorem*{question*}{Question}
\newtheorem{example}{Example}

\begin{document}

\title[Nadler-Quinn problem, simplicial systems]{The Nadler-Quinn problem for\\simplicial inverse systems}
\author{Andrea Ammerlaan \and Ana Anu\v{s}i\'{c} \and Logan C. Hoehn}
\date{\today}

\address{Nipissing University, Department of Computer Science \& Mathematics, 100 College Drive, Box 5002, North Bay, Ontario, Canada, P1B 8L7}
\email{ajammerlaan879@my.nipissingu.ca}
\email{anaa@nipissingu.ca}
\email{loganh@nipissingu.ca}

\dedicatory{Dedicated to the memory of Piotr Minc}

\thanks{This work was supported by NSERC grant RGPIN-2019-05998}

\subjclass[2020]{Primary 54F15, 54C25; Secondary 54F50}
\keywords{Plane embeddings; accessible point; arc-like continuum}

\begin{abstract}
We show that if $X$ is an arc-like continuum which can be represented as an inverse limit of a simplicial inverse system on arcs, then for every point $x \in X$ there is a plane embedding of $X$ in which $x$ is accessible.  This answers a special case of the Nadler-Quinn question from 1972.
\end{abstract}

\maketitle

\section{Introduction}
\label{sec:intro}

In this paper we continue the study of plane embeddings of arc-like continua and their accessible sets from \cite{ammerlaan-anusic-hoehn2023}.  A \emph{continuum} is a compact, connected, metric space, and a continuum is \emph{arc-like} if it homeomorphic to an inverse limit $\varprojlim \left \langle [-1,1],f_n \right \rangle$, where $f_n \colon [-1,1] \to [-1,1]$ are continuous functions which can be taken to be piecewise-linear and onto.  We are interested in answering the question of Nadler and Quinn from 1972, which first appeared in \cite[p. 229]{nadler1972} and \cite{nadler-quinn1972}, and has subsequently appeared over the years in works by different authors and in various forms (see Problem 140 in \cite{lewis1983}, Question 16 in \cite{problems2002}, and Problem 1, Problem 49, and Problem 51 in \cite{problems2018}).  The question of Nadler and Quinn asks whether for every arc-like continuum $X$, and every $x \in X$, there exists a plane embedding $\Omega \colon X \to \mathbb{R}^2$ such that $\Omega(x)$ is an \emph{accessible} point of $\Omega(X)$, i.e.\ such that there is an arc $A \subset \mathbb{R}^2$ with $A \cap \Omega(X) = \{\Omega(x)\}$.

As pointed out in \cite{ammerlaan-anusic-hoehn2023}, by a simple reduction one can see that this question of Nadler and Quinn is equivalent to the following:

\begin{question*}[Question~1 of \cite{ammerlaan-anusic-hoehn2023}]
If $f_n \colon [-1,1] \to [-1,1]$, $n = 1,2,\ldots$, are piecewise-linear maps such that $f_n(0) = 0$ for each $n$, does there exist an embedding of $X = \varprojlim \left \langle [-1,1],f_n \right \rangle$ into $\mathbb{R}^2$ for which the point $\langle 0,0,\ldots \rangle \in X$ is accessible?
\end{question*}

\noindent The main result of \cite[Theorem~7.3]{ammerlaan-anusic-hoehn2023} shows that if for every $n \geq 1$,
\begin{enumerate}
\item the radial contour factors of $f_n$ and $f_n\circ f_{n+1}$ are equal, and
\item $f_n$ has no contour twins,
\end{enumerate}
then $\varprojlim \left \langle [-1,1],f_n \right \rangle$ can be embedded in the plane such that $\langle 0,0,\ldots \rangle$ is accessible.  For the definitions of radial contour factors and contour twins see \cite{ammerlaan-anusic-hoehn2023} and Section~\ref{sec:rad deps} below.  The main result of this paper answers Question~2 of \cite{ammerlaan-anusic-hoehn2023}, by eliminating condition (2) above, see Theorem~\ref{thm:same contour}.  As such, we will not discuss the notion of contour twins in this paper.  As a corollary, we show that if $\left \langle [-1,1],f_n \right \rangle$ is a simplicial inverse system, then for every point $x \in X = \varprojlim \left \langle [-1,1],f_n \right \rangle$ there is a plane embedding of $X$ in which $x$ is accessible.  This special case of the Nadler-Quinn question was inspired by related questions of P.\ Minc in \cite[p.296]{problems2018}.  See Section~\ref{sec:simplicial} below for the definition of a simplicial inverse system.

The paper is structured as follows.  In Section~\ref{sec:prelim} we give general preliminaries and recall notions and properties of (radial) departures and (radial) contour factors from \cite{ammerlaan-anusic-hoehn2023}.  In Section~\ref{sec:bridging} we state ``bridging lemmas'', which will allow us to change maps on certain intervals while preserving the composition with the radial contour factor.  In Section~\ref{sec:same contour} we show how the bridging lemmas can be used to answer the Nadler-Quinn question in positive if the radial contour factors of bonding maps are preserved with the composition, see Theorem~\ref{thm:same contour}.  In Section~\ref{sec:simplicial} we apply Theorem~\ref{thm:same contour} to give a positive answer to the question of Nadler and Quinn for arc-like continua which are inverse limits of simplicial systems, see Theorem~\ref{thm:simplicial}.

\section{Preliminaries}
\label{sec:prelim}

By a \emph{map} we mean a continuous function.  Given a map $f$, we denote the domain of $f$ by $\operatorname{dom} f$.  If $f \colon [-1,1] \to [-1,1]$ is a map, we call it \emph{piecewise-linear} if there exists $n \in \mathbb{N}$ and $-1 = c_0 < c_1 < \cdots < c_n = 1$ such that $f$ is linear on $[c_{i-1},c_{i}]$ for each $i \in \{1,\ldots,n\}$.

A \emph{continuum} is a compact, connected, metric space.  For a sequence $X_n$, $n \geq 1$, of continua, and a sequence $f_n \colon X_{n+1} \to X_{n}$, $n \geq 1$, of maps, the sequence
\[ \left \langle X_{n},f_n \right \rangle_{n\geq 1} = \left \langle X_1,f_1,X_2,f_2,X_3,\ldots \right \rangle \]
is called an \emph{inverse system}.  The spaces $X_n$ are called \emph{factor spaces}, and the maps $f_n$ are called \emph{bonding maps}.  The \emph{inverse limit} of an inverse system $\left \langle X_n,f_n \right \rangle_{n\geq 1}$ is the space
\[ \varprojlim \langle X_n,f_n \rangle_{n\geq 1} = \{(x_1,x_2,\ldots): f_n(x_{n+1}) = x_n \textrm{ for each } n \geq 1\} \subseteq \prod_{n \geq 1} X_n ,\]
equipped with the subspace topology inherited from the product topology on $\prod_{n \geq 1} X_n$.  It is also a continuum.  When the index set is clear, we will use a shortened notation $\varprojlim \left \langle X_n,f_n \right \rangle$.  Note that the inverse limit representation of a continuum is not unique, for example:
\begin{itemize}
\item (Dropping finitely many coordinates) If $n_0 \geq 1$, then $\varprojlim \left \langle X_n,f_n \right \rangle_{n \geq 1} \approx \varprojlim \left \langle X_n,f_n \right \rangle_{n \geq n_0}$.
\item (Composing bonding maps) If $1 = n_1 < n_2 < n_3 < \ldots$ is an increasing sequence of integers, then $\varprojlim \left \langle X_n,f_n \right \rangle_{n \geq 1} \approx \varprojlim \left \langle X_{n_i},f_{n_i} \circ \cdots \circ f_{n_{i+1}-1} \right \rangle_{i \geq 1}$.
\end{itemize}

A continuum $X$ is called {\em arc-like} if $X \approx \varprojlim \left \langle [-1,1],f_n \right \rangle_{n \geq 1}$ for some sequence of maps $f_n \colon [-1,1] \to [-1,1]$, $n \geq 1$.  By \cite{brown1960}, we can without loss of generality assume that these maps $f_n$ are all piecewise-linear.

\subsection{Radial departures and radial contour factors}
\label{sec:rad deps}

In this subsection, we recall the definitions of radial departures and radial contour factorizations from \cite{ammerlaan-anusic-hoehn2023}, and some of the basic results about them that we will make use of in this paper.  We use letters for these definitions and results, to distinguish them from the new results of this paper in later sections, which are labelled with numbers.  We begin with radial departures.

\begin{defn*}[Definition~4.2 of \cite{ammerlaan-anusic-hoehn2023}]
Let $f \colon [-1,1] \to [-1,1]$ be a map with $f(0) = 0$.  A \emph{radial departure} of $f$ is a pair $\langle x_1,x_2 \rangle$ such that $-1 \leq x_1 < 0 < x_2 \leq 1$ and either:
\begin{enumerate}[label=(\arabic{*})]
\item \label{pos dep} $f((x_1,x_2)) = (f(x_1),f(x_2))$; or
\item \label{neg dep} $f((x_1,x_2)) = (f(x_2),f(x_1))$.
\end{enumerate}
We say a radial departure $\langle x_1,x_2 \rangle$ of $f$ is \emph{positively oriented} (or a \emph{positive radial departure}) if \ref{pos dep} holds, and $\langle x_1,x_2 \rangle$ is \emph{negatively oriented} (or a \emph{negative radial departure}) if \ref{neg dep} holds.
\end{defn*}

\begin{propalph}[Proposition~4.5 of \cite{ammerlaan-anusic-hoehn2023}]
\label{prop:alt dep nested}
Let $f \colon [-1,1] \to [-1,1]$ be a map with $f(0) = 0$.  Suppose $\langle x_1,x_2 \rangle$ and $\langle x_1',x_2' \rangle$ are radial departures of $f$ with opposite orientations.  Then either
\[ x_1 < x_1' < 0 < x_2' < x_2 \quad \textrm{or} \quad x_1' < x_1 < 0 < x_2 < x_2' .\]
\end{propalph}

Though it is not explicitly stated in Proposition~\ref{prop:alt dep nested}, it is a simple consequence that if $\langle x_1,x_2 \rangle$ is a positive radial departure of $f$ and $\langle x_1',x_2' \rangle$ is a negative radial departure of $f$, then either
\[ f(x_1) < f(x_2') < 0 < f(x_1') < f(x_2) \quad \textrm{or} \quad f(x_2') < f(x_1) < 0 < f(x_2) < f(x_1') .\]

Thus the presence of radial departures of opposite orientations captures the essence of a ``zig-zag'' pattern in the graph of $f$ around $0$.  This is the key intuition behind the following result, which is the main application of radial departures.

\begin{propalph}[Proposition~5.2 of \cite{ammerlaan-anusic-hoehn2023}]
\label{prop:embed 0 accessible}
Let $f_n \colon [-1,1] \to [-1,1]$, $n = 1,2,\ldots$, be maps with $f_n(0) = 0$ for each $n$.  Suppose that for each $n$, all radial departures of $f_n$ have the same orientation.  Then there exists an embedding of $X = \varprojlim \left \langle [-1,1], f_n \right \rangle$ into $\mathbb{R}^2$ for which the point $\langle 0,0,\ldots \rangle \in X$ is accessible.
\end{propalph}

\begin{propalph}[Proposition~4.6 of \cite{ammerlaan-anusic-hoehn2023}]
\label{prop:comp dep}
Let $f,g \colon [-1,1] \to [-1,1]$ be maps with $f(0) = g(0) = 0$, and let $x_1,x_2$ be such that $-1 \leq x_1 < 0 < x_2 \leq 1$.  Then $\langle x_1,x_2 \rangle$ is a positive (respectively, negative) radial departure of $f \circ g$ if and only if either:
\begin{enumerate}
\item $\langle x_1,x_2 \rangle$ is a positive radial departure of $g$ and $\langle g(x_1),g(x_2) \rangle$ is a positive (respectively, negative) radial departure of $f$; or
\item $\langle x_1,x_2 \rangle$ is a negative radial departure of $g$ and $\langle g(x_2),g(x_1) \rangle$ is a negative (respectively, positive) radial departure of $f$.
\end{enumerate}
\end{propalph}

\medskip
Next, we recall the radial contour factorization of a function $f \colon [-1,1] \to [-1,1]$ with $f(0) = 0$.  We begin with the ``one-sided'' contour factorization of a map $f \colon [0,1] \to [-1,1]$.

\begin{defn*}[Definition~3.1 of \cite{ammerlaan-anusic-hoehn2023}]
Let $f \colon [0,1] \to [-1,1]$ be a map with $f(0) = 0$.  A \emph{departure} of $f$ is a number $x > 0$ such that $f(x) \notin f([0,x))$.  We say a departure $x$ of $f$ is \emph{positively oriented} (or a \emph{positive departure}) if $f(x) > 0$, and $x$ is \emph{negatively oriented} (or a \emph{negative departure}) if $f(x) < 0$.

A \emph{contour point} of $f$ is a departure $\alpha$ such that for any departure $x$ of $f$ with $x > \alpha$, there exists a departure $y$ of $f$, of orientation opposite to that of $\alpha$, such that $\alpha < y \leq x$.
\end{defn*}

If $f$ is not constant, then it has departures, and therefore has at least one contour point.  For convenience we may also count $\alpha = 0$ as a contour point of $f$.

\begin{defn*}[Definition~3.2 of \cite{ammerlaan-anusic-hoehn2023}]
Let $f \colon [0,1] \to [-1,1]$ be a non-constant piecewise-linear map with $f(0) = 0$.  The \emph{contour factor} of $f$ is the piecewise-linear map $t_f \colon [0,1] \to [-1,1]$ defined as follows:

Let $0 = \alpha_0 < \alpha_1 < \alpha_2 < \cdots < \alpha_n$ be the contour points of $f$.  Then $t_f \left( \frac{i}{n} \right) = f(\alpha_i)$ for each $i = 0,\ldots,n$, and $t_f$ is defined to be linear in between these points.

A \emph{meandering factor} of $f$ is any (piecewise-linear) map $s \colon [0,1] \to [0,1]$ such that $s(0) = 0$ and $f = t_f \circ s$.
\end{defn*}

For the next two Definitions, let $\mathsf{r} \colon [0,1] \to [-1,0]$ be the function $\mathsf{r}(x) = -x$

\begin{defn*}[Definition~4.1 of \cite{ammerlaan-anusic-hoehn2023}]
Let $f \colon [-1,1] \to [-1,1]$ be a map with $f(0) = 0$.
\begin{itemize}
\item A \emph{right departure} of $f$ is a number $x > 0$ such that $x$ is a departure of $f {\restriction}_{[0,1]}$.
\item A \emph{left departure} of $f$ is a number $x < 0$ such that $-x$ is a departure of $f {\restriction}_{[-1,0]} \circ \mathsf{r}$; i.e.\ a number $x < 0$ such that $f(x) \notin f((x,0])$.
\end{itemize}
If $x$ is either a right departure or a left departure of $f$, then we say $x$ is \emph{positively oriented} if $f(x) > 0$, and $x$ is \emph{negatively oriented} if $f(x) < 0$.
\begin{itemize}
\item A \emph{right contour point} of $f$ is a number $\alpha > 0$ such that $\alpha$ is a contour point of $f {\restriction}_{[0,1]}$.
\item A \emph{left contour point} of $f$ is a number $\beta < 0$ such that $-\beta$ is a contour point of $f {\restriction}_{[-1,0]} \circ \mathsf{r}$.
\end{itemize}
\end{defn*}

In the remainder of this paper we consider maps $f \colon [-1,1] \to [-1,1]$ satisfying $f(0) = 0$, and for which $f {\restriction}_{[-1,0]}$ and $f {\restriction}_{[0,1]}$ are both non-constant (we will tacitly assume this latter condition throughout the remainder of this paper).

\begin{defn*}[Definition~6.1 of \cite{ammerlaan-anusic-hoehn2023}]
Let $f \colon [-1,1] \to [-1,1]$ be a piecewise-linear map with $f(0) = 0$.  The \emph{radial contour factor} of $f$ is the piecewise-linear map $t_f \colon [-1,1] \to [-1,1]$ such that:
\begin{enumerate}
\item $t_f {\restriction}_{[0,1]}$ is the contour factor of $f {\restriction}_{[0,1]}$, and;
\item $t_f {\restriction}_{[-1,0]} \circ \mathsf{r}$ is the contour factor of $f {\restriction}_{[-1,0]} \circ \mathsf{r}$.
\end{enumerate}
A \emph{radial meandering factor} of $f$ is a (piecewise-linear) map $s \colon [-1,1] \to [-1,1]$ which is \emph{sign-preserving} (i.e.\ $s(x) \geq 0$ for all $x \geq 0$ and $s(x) \leq 0$ for all $x \leq 0$) and such that $f = t_f \circ s$.
\end{defn*}

As pointed out in \cite{ammerlaan-anusic-hoehn2023}, there always exists at least one radial meandering factor of any given $f$, but it is not necessarily unique.  Much of the work in this paper is devoted to producing and selecting alternative maps $s \colon [-1,1] \to [-1,1]$, which are not necessarily sign-preserving, such that $s(0) = 0$ and $f = t_f \circ s$.

\begin{lemalph}[Lemma~6.2 of \cite{ammerlaan-anusic-hoehn2023}]
\label{lem:s match}
Let $f \colon [-1,1] \to [-1,1]$ be a piecewise-linear map with $f(0) = 0$, let $t_f$ be the radial contour factor of $f$, and let $\langle x_1,x_2 \rangle$ be a radial departure of $f$.  Then there exists a radial departure $\langle y_1,y_2 \rangle$ of $t_f$ such that for any map $s \colon [-1,1] \to [-1,1]$ with $s(0) = 0$ and $f = t_f \circ s$:
\begin{enumerate}
\item $\langle x_1,x_2 \rangle$ is a positive radial departure of $s$; and
\item $s(x_1) = y_1$ and $s(x_2) = y_2$.
\end{enumerate}
\end{lemalph}

\begin{propalph}[see Corollary~6.5 and Proposition~6.6(2) of \cite{ammerlaan-anusic-hoehn2023}]
\label{prop:comp same dep}
Let $f_1,f_2 \colon [-1,1] \to [-1,1]$ be maps with $f_1(0) = f_2(0) = 0$, and suppose $f_1$ and $f_1 \circ f_2$ have the same radial contour factors.  Then for any negative radial departure $\langle x_1,x_2 \rangle$ of $f_2$, $\langle f_2(x_2),f_2(x_1) \rangle$ is not a radial departure of $f_1$.  In fact, if $\langle y_1,y_2 \rangle$ is any radial departure of $f_1$ then either
\[ f_2(x_2) < y_1 < 0 < y_2 < f_2(x_1) \quad \textrm{or} \quad y_1 < f_2(x_2) < 0 < f_2(x_1) < y_2 .\]
\end{propalph}

\section{Bridging lemmas}
\label{sec:bridging}

Throughout this section, let $f,s \colon [-1,1] \to [-1,1]$ be piecewise-linear maps with $f(0) = s(0) = 0$, let $t = t_f$ be the radial contour factor of $f$, and suppose $f = t \circ s$.  We remark that $s$ is not assumed to be sign-preserving, so is not necessarily a radial meandering factor.  The main results of this section are the ``Bridging Lemmas'' (Lemmas~\ref{lem:bridging I} and \ref{lem:bridging II}), which enable us to change the map $s$ on certain intervals while preserving the properties $s(0) = 0$ and $f = t \circ s$.

In the following Definition, the name ``liftable range'' alludes to the property established in Lemma~\ref{lem:stay right} below, which states roughly that when $s$ maps an interval to a liftable range, then we may modify $s$ on that interval so that it is non-negative throughout the interval.

\begin{defn}
\label{defn:liftable}
A \emph{liftable range} is an interval $[y^-,y^+]$ with $0 \in [y^-,y^+] \subseteq [-1,1]$, and such that:
\begin{enumerate}
\item $t([0,1]) \supseteq t([y^-,0])$; and
\item There does not exist a radial departure $\langle y_1,y_2 \rangle$ of $t$ such that $y^- \leq y_1$ and $y^+ < y_2$.
\end{enumerate}
\end{defn}

\begin{lem}
\label{lem:liftable condition}
Let $y^-,y^+ \in [-1,1]$ with $y^- < 0 \leq y^+$.  If there exist $0 \leq x < x' \leq 1$ such that $x$ is a positive right departure of $s$ (or $x = 0$) with $s(x) = y^+$, and $s([x,x']) = [y^-,y^+]$, then $[y^-,y^+]$ is a liftable range.
\end{lem}

\begin{proof}
To see that $t([0,1]) \supseteq t([y^-,0])$, let $s_0$ be a radial meandering factor of $f$, so that $t \circ s_0 = f$ and $s_0([0,1]) = [0,1]$, hence $s_0([x,x']) \subseteq [0,1]$.  Then $t([0,1]) \supseteq t \circ s_0([x,x']) = t \circ s([x,x']) = t([y^-,y^+]) \supseteq t([y^-,0])$.

For the second property, suppose for a contradiction that $\langle y_1,y_2 \rangle$ is a radial departure of $t$ with $y^- \leq y_1$ and $y^+ < y_2$.  By Lemma~\ref{lem:s match}, there is a positive radial departure $\langle x_1,x_2 \rangle$ of $s$ such that $s(x_1) = y_1$ and $s(x_2) = y_2$.  Since $y^+ < y_2$, we have $s(z) < y_2$ for all $z \in [x,x']$, and since $x$ is a right departure of $s$ (or $x=0$) with $s(x) = y^+$, we have $s(z) < y_2$ for all $z \in [0,x]$.  Thus $s(z) < y_2$ for all $z \in [0,x']$, and hence $x_2 > x'$, meaning $[x,x'] \subset (x_1,x_2)$.  But now $y^- \in s([x,x']) \subset s((x_1,x_2))$, while $y_1 = s(x_1) \notin s((x_1,x_2))$, a contradiction since $y^- \leq y_1 < 0$.
\end{proof}

The following simple observation is an immediate consequence of the definition of a liftable range.

\begin{lem}
\label{lem:liftable comparison}
Let $[y^-,y^+]$ be a liftable range.  If $y^- \leq w^- < 0 \leq y^+ \leq w^+ \leq 1$, then $[w^-,w^+]$ is a liftable range as well.
\end{lem}

In the ``Bridging Lemmas'' below, we will make use of the following function $L$.

\begin{defn}
\label{defn:L}
Given $y < 0$ such that $t([0,1]) \supseteq t([y,0])$, define $L(y)$ to be the minimal value such that $L(y) > 0$ and $t([0,L(y)]) \supseteq t([y,0])$.
\end{defn}

Observe that $L(y)$ is a right departure of $t$.  This function $L$ is monotone: if $y \leq y' < 0$, then $L(y') \leq L(y)$.  $L$ is not continuous in general.

When we say $I$ is a closed interval with endpoints $a$ and $b$, we mean either $I = [a,b]$ (if $a < b$) or $I = [b,a]$ (if $b < a$).

\begin{lem}
\label{lem:stay right}
Let $[y^-,y^+]$ be a liftable range, and let $I \subseteq [-1,1]$ be a closed interval with endpoints $a$ and $b$, such that $s(I) = [y^-,y^+]$ and $s(a) = y^+$.  Then there exists a map $\hat{s} \colon I \to [0,1]$ such that:
\begin{enumerate}[label=(\arabic{*})]
\item $t \circ \hat{s}(x) = f(x)$ for all $x \in I$;
\item $\hat{s}(a) = y^+$;
\item $\hat{s}(x) \geq s(x)$ for all $x \in I$;
\item $\max \{\hat{s}(x): x \in I$ and $\hat{s}(x) \neq s(x)\} = L(y^-)$; and
\item If $s(b) \leq L(y^-)$, then $\hat{s}(b)$ is the largest value $y \in [0,L(y^-)]$ such that $t(y) = f(b)$.
\end{enumerate}
\end{lem}

Note that $\hat{s}(b)$ is uniquely determined by conditions (4) and (5); namely, if $s(b) > L(y^-)$ then $\hat{s}(b) = s(b)$ according to condition (4), and if $s(b) \leq L(y^-)$ then $\hat{s}(b)$ is given by condition (5).  The value of $\hat{s}(b)$ depends only on $t$ and the values $y^-$ and $s(b)$.

See Figure~\ref{fig:liftable range} for an illustration of the scenario described in Lemma~\ref{lem:stay right}.

\begin{figure}
\begin{center}
\includegraphics{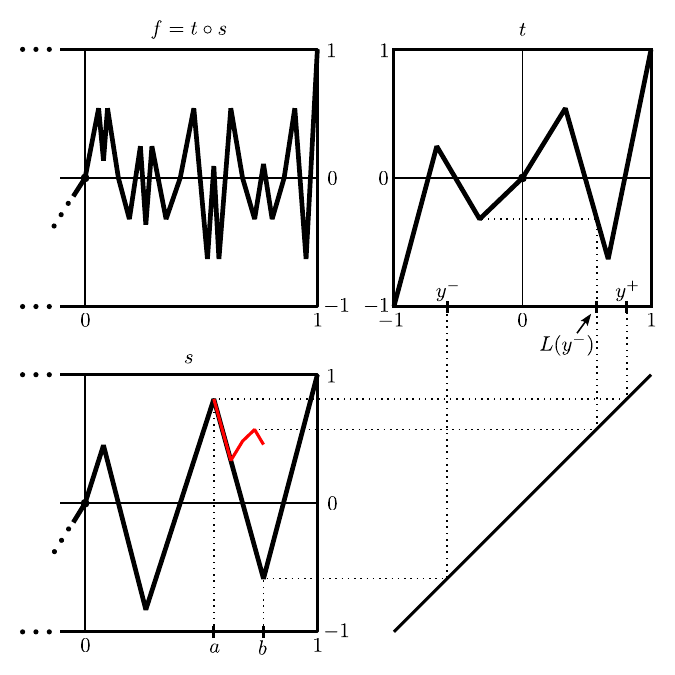}
\end{center}

\caption{An example of a liftable range $[y^-,y^+]$ and interval $I = [a,b]$ as described in Lemma~\ref{lem:stay right}, together with the map $\hat{s}$ (in red).}
\label{fig:liftable range}
\end{figure}

\begin{proof}
Let $\gamma$ be the largest right contour point of $t$ such that $\gamma < L(y^-)$ (or $\gamma = 0$ if there is no such right contour point).  Observe that $t$ is one-to-one on $[\gamma,L(y^-)]$.

\begin{claim}
\label{claim:gamma L covers}
$t([\gamma,L(y^-)]) = t([y^-,L(y^-)])$.
\end{claim}

\begin{proof}[Proof of Claim~\ref{claim:gamma L covers}]
\renewcommand{\qedsymbol}{\textsquare (Claim~\ref{claim:gamma L covers})}
Clearly $t([\gamma,L(y^-)]) \subseteq t([y^-,L(y^-)])$ as $[\gamma,L(y^-)] \subseteq [y^-,L(y^-)]$.  Also, by definition of $L(y^-)$, $t([0,L(y^-)]) \supseteq t([y^-,0])$, meaning $t([0,L(y^-)]) \supseteq t([y^-,L(y^-)])$.  Therefore to prove the Claim, it suffices to prove that $t([\gamma,L(y^-)]) \supseteq t([0,\gamma])$.  Clearly we may assume $\gamma \neq 0$.

For simplicity, suppose that $t(\gamma) > 0$ (the case $t(\gamma) < 0$ is similar).  Let $\gamma'$ be the right contour point of $t$ immediately preceding $\gamma$ (or $\gamma' = 0$ if there is no such right contour point), so that $t(\gamma') \leq 0$ and $t([0,\gamma]) = [t(\gamma'),t(\gamma)]$.  Since $L(y^-)$ is a right departure of $t$ and there are no right contour points between $\gamma$ and $L(y^-)$, we must have that $L(y^-)$ is a negative right departure of $t$ and $t(L(y^-)) < t(\gamma')$, and so $t([\gamma,L(y^-)]) = [t(L(y^-)),t(\gamma)] \supseteq [t(\gamma'),t(\gamma)] = t([0,\gamma])$, as desired.
\end{proof}

We define $\hat{s}$ in two cases, depending on whether $y^+ \geq \gamma$ or $y^+ < \gamma$.

\medskip
\textbf{Case~1:} Suppose $y^+ \geq \gamma$.

Given $x \in I$, define $\hat{s}(x)$ as follows.  If $s(x) > L(y^-)$, then let $\hat{s}(x) = s(x)$.  If $s(x) \leq L(y^-)$, then let $\hat{s}(x)$ be the (unique) point in $[\gamma,L(y^-)]$ such that $t \circ \hat{s}(x) = t \circ s(x)$ (this point exists by Claim~\ref{claim:gamma L covers}).  It is straightforward to see that this function $\hat{s}$ is continuous and satisfies the properties (1)--(5) of the Lemma.

\medskip
\textbf{Case~2:} Suppose $y^+ < \gamma$.

Let $\gamma_1 < \gamma_2 < \cdots < \gamma_k = \gamma$ enumerate all the right contour points of $t$ strictly between $y^+$ and $L(y^-)$.  If $\gamma_1$ is not the first right contour point of $t$, we consider also the right contour point $\gamma_0$ of $t$ immediately preceding $\gamma_1$; if $\gamma_1$ is the first right contour point of $t$, then let $\gamma_0 = 0$.  Also, for convenience, denote $\gamma_{k+1} = L(y^-)$ (but note that this is not necessarily a right contour point of $t$).  Then
\[ 0 \leq \gamma_0 < \gamma_1 < \gamma_2 < \cdots < \gamma_k = \gamma < \gamma_{k+1} = L(y^-) \]
and for each $1 \leq i \leq k+1$, we have $t([0,\gamma_i]) = t([\gamma_{i-1},\gamma_i])$.  See Figure~\ref{fig:stay right} for an illustration.

\begin{figure}
\begin{center}
\includegraphics{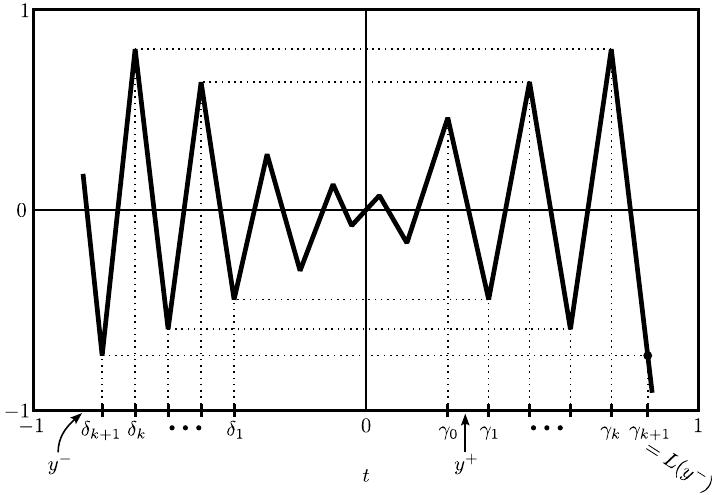}
\end{center}

\caption{A sample illustration of part of the map $t$ for Case~2 in the proof of Lemma~\ref{lem:stay right}, with the points $y^-$ and $y^+$ indicated, as well as the points $\gamma_0,\gamma_1,\ldots,\gamma_k,\gamma_{k+1}$ and $\delta_1,\ldots,\delta_k,\delta_{k+1}$ from Claim~\ref{claim:twins}.}
\label{fig:stay right}
\end{figure}

\begin{claim}
\label{claim:twins}
There exist points $\delta_1,\ldots,\delta_{k+1}$ with
\[ y^- \leq \delta_{k+1} < \delta_k < \cdots < \delta_2 < \delta_1 < 0 \]
and such that for each $1 \leq i \leq k+1$,
\begin{enumerate}
\item $\delta_i$ is a left departure of $t$ with $t(\delta_i) = t(\gamma_i)$; and
\item $t([\delta_i,y^+]) \subseteq t([\gamma_{i-1},\gamma_i])$.
\end{enumerate}
\end{claim}

In fact, the two sets in condition (2) will be equal, but we will only use the inclusion stated.

\begin{proof}[Proof of Claim~\ref{claim:twins}]
\renewcommand{\qedsymbol}{\textsquare (Claim~\ref{claim:twins})}
First observe that for each $1 \leq i \leq k+1$, since $y^+ < \gamma_i$, we have $t([0,y^+]) \subseteq t([0,\gamma_i]) = t([\gamma_{i-1},\gamma_i])$.  In light of this, for condition (2) it will suffice to show that $t([\delta_i,0]) \subseteq t([0,\gamma_i])$.

We construct the points $\delta_i$ recursively, starting at $i = k+1$ and proceeding down to $i = 1$.

For $i = k+1$, there exists a point $\delta \in [y^-,0)$ with $t(\delta) = t(\gamma_{k+1}) = t(L(y^-))$ by minimality of $L(y^-)$.  Let $\delta_{k+1}$ be the maximal such $\delta$, so that $\delta_{k+1}$ is a left departure of $t$.  Now $t([\delta_{k+1},0]) \subseteq t([y^-,0]) \subseteq t([0,L(y^-)]) = t([0,\gamma_{k+1}])$, thus condition (2) is satisfied for $i = k+1$.  We point out that $\delta_{k+1}$ must either be a left contour point of $t$, or $\delta_{k+1} = y^-$, though we will not use this observation in the argument here.

Now suppose $1 \leq i \leq k$, and that $\delta_{i+1}$ has been constructed.  Notice that $t([0,\gamma_i])$ is a closed interval with endpoints $t(\gamma_i)$ and $t(\gamma_{i-1})$, and $t(\delta_{i+1}) = t(\gamma_{i+1}) \notin t([0,\gamma_i])$.  Also, $t(\delta_{i+1}) \notin t((\delta_{i+1},0])$ since $\delta_{i+1}$ is a left departure of $t$, and $t(\gamma_i) \notin t([0,\gamma_i))$ since $\gamma_i$ is a right departure of $t$.  So there must exist a point $\delta \in (\delta_{i+1},0)$ with $t(\delta) = t(\gamma_i)$, for otherwise $\langle \delta_{i+1},\gamma_i \rangle$ would be a radial departure of $t$ with $y^- \leq \delta_{i+1}$ and $y^+ < \gamma_i$, contradicting the assumption that $[y^-,y^+]$ is a liftable range.  Let $\delta_i \in (\delta_{i+1},0)$ be maximal such that $t(\delta_i) = t(\gamma_i)$, so that $\delta_i$ is a left departure of $t$.  We point out that $\delta_i$ must in fact be a contour point of $t$, though we will not use this observation in the argument here.

To verify condition (2), suppose for simplicity that $t(\gamma_i) > 0$, so that $t(\gamma_{i-1}) \leq 0$ (the case $t(\gamma_i) < 0$ is similar).  Then $t([0,\gamma_i]) = [t(\gamma_{i-1}),t(\gamma_i)]$.  Suppose for a contradiction that there exists $y \in [\delta_i,0]$ such that $t(y) \notin [0,t(\gamma_i)]$.  We may assume that $y$ is a left departure of $t$.  Since $\delta_i$ is a left departure of $t$ with $t(\delta_i) = t(\gamma_i)$, and $y \in [\delta_i,0]$, it follows that $t(y) < t(\gamma_{i-1})$.  Now $t(\gamma_i) \notin t([0,\gamma_i))$ since $\gamma_i$ is a right departure of $t$, and $t(\gamma_i) = t(\delta_i) \notin t((\delta_i,0]) \supseteq t((y,0])$ since $\delta_i$ is a left departure of $t$.  Further, $t(y) \notin t((y,0])$ since $y$ is a left departure of $t$, and $t(y) \notin [t(\gamma_{i-1}),t(\gamma_i)] = t([0,\gamma_i])$.  But this means $\langle y,\gamma_i \rangle$ is a (positive) radial departure of $t$ with $y^- \leq y$ and $y^+ < \gamma_i$, contradicting the assumption that $[y^-,y^+]$ is a liftable range.  Therefore $t([\delta_i,0]) \subseteq t([0,\gamma_i])$, as desired.
\end{proof}

For simplicity we assume that $a < b$; the argument can easily be adapted in the event that $b < a$.

For each $1 \leq i \leq k+1$, let $x_i \in [a,b]$ be minimal such that $s(x_i) = \delta_i$.  Also, let $x_0 = a$.  Then
\[ a = x_0 < x_1 < x_2 < \cdots < x_k < x_{k+1} \leq b, \]
and
\begin{enumerate}
	\item[(a)] $s([x_{i-1},x_i]) \subseteq [\delta_i,y^+]$ for each $1 \leq i \leq k+1$, and
	\item[(b)] $s([x_{k+1},b]) \subseteq [\delta_{k+1},y^+]\subseteq [y^{-},L(y^{-})]$.
\end{enumerate}

Define $\hat{s}$ as follows.  If $x \in [x_{i-1},x_i]$ for some $1 \leq i \leq k+1$, let $\hat{s}(x)$ be the (unique) point $y \in [\gamma_{i-1},\gamma_i]$ such that $t(y) = t \circ s(x)$ (this point exists by Claim~\ref{claim:twins}(2) and (a)).  If $x \in [x_{k+1},b]$ let $\hat{s}(x)$ be the (unique) point $y \in [\gamma_k,L(y^-)]$ such that $t(y) = t \circ s(x)$ (this point exists by Claim~\ref{claim:gamma L covers} and (b)).  

It is straightforward to see that this function $\hat{s}$ is continuous, with $\hat{s}(x_i) = \gamma_i$ for each $1 \leq i \leq k+1$.  As for the properties (1)--(5) of the Lemma, properties (1), (2), and (3) are immediate from the construction, and (5) follows from (1) and the fact that $\hat{s}(b) \in [\gamma,L(y^-)]$ and $t$ is one-to-one on $[\gamma,L(y^-)]$.  For (4), note that $\hat s(x)\leq L(y^{-})$ for all $x\in [a,b]$ by construction. Also, $\hat s(x_{k+1})=L(y^{-})$, and $s(x_{k+1})=\delta_{k+1}\neq \hat s(x_{k+1})$.
\end{proof}

For the next Lemmas, we adopt the following notation: let $0 < \alpha_1 < \alpha_2 < \ldots < \alpha_n \leq 1$ enumerate all the right contour points of $s$, and let $0 > \beta_1 > \beta_2 > \ldots > \beta_m \geq -1$ enumerate all the left contour points of $s$.  Also, denote $\alpha_0 = \beta_0 = 0$.

\begin{lem}[Bridging in I]
\label{lem:bridging I}
Let $\alpha_i$ be a negative right contour point of $s$, so that $\alpha_{i-1}$ is a positive right contour point of $s$ or $\alpha_{i-1} = 0$.  Then there exists $\hat{s}_{\alpha_i} \colon [\alpha_{i-1},\alpha_{i+1}] \to [0,1]$ (or $\hat{s}_{\alpha_i} \colon [\alpha_{i-1},1] \to [0,1]$ if $\alpha_i$ is the last negative right contour point of $s$) such that:
\begin{enumerate}[label=(\arabic{*})]
\item $t \circ \hat{s}_{\alpha_i}(x) = f(x)$ for all $x \in \operatorname{dom} \hat{s}_{\alpha_i}$;
\item $\hat{s}_{\alpha_i}(\alpha_{i-1}) = s(\alpha_{i-1})$ and, if $\alpha_i$ is not the last right contour point of $s$, then $\hat{s}_{\alpha_i}(\alpha_{i+1}) = s(\alpha_{i+1})$;
\item $\hat{s}_{\alpha_i}(x) \geq s(x)$ for all $x \in \operatorname{dom} \hat{s}_{\alpha_i}$; and
\item $\max \{\hat{s}_{\alpha_i}(x): \hat{s}_{\alpha_i}(x) \neq s(x)\} = L(s(\alpha_i))$, and there exists $x \in [\alpha_{i-1},\alpha_i]$ such that $\hat{s}_{\alpha_i}(x) = L(s(\alpha_i))$ (recall that the function $L$ is defined in Definition~\ref{defn:L}).
\end{enumerate}
\end{lem}

In Figure~\ref{fig:bridging I}, we give a few illustrations of the map $\hat{s}_{\alpha_i}$ from Lemma~\ref{lem:bridging I} for sample functions $s$.

\begin{figure}
\begin{center}
\begin{subfigure}{0.45\linewidth}
\includegraphics{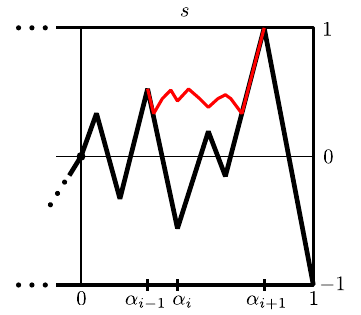}
\caption{$\alpha_i$ is not last negative right contour point, so that $\operatorname{dom} \hat{s}_{\alpha_i} = [\alpha_{i-1},\alpha_{i+1}]$.}
\label{fig:bridging I a}
\end{subfigure}\\
\begin{subfigure}{0.45\linewidth}
\includegraphics{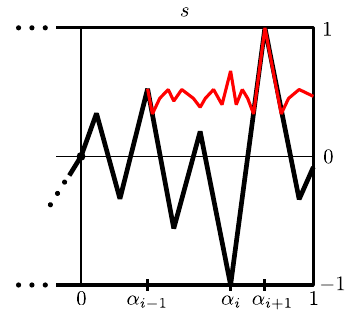}
\caption{$\alpha_i$ is last negative right contour point, so that $\operatorname{dom} \hat{s}_{\alpha_i} = [\alpha_{i-1},1]$.}
\label{fig:bridging I b}
\end{subfigure}
\hfill
\begin{subfigure}{0.45\linewidth}
\includegraphics{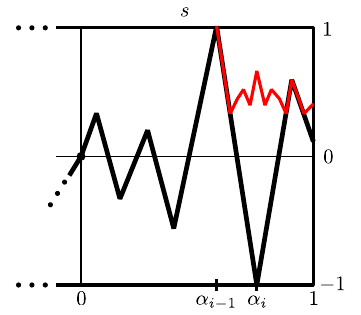}
\caption{$\alpha_i$ is last right contour point, so that $\operatorname{dom} \hat{s}_{\alpha_i} = [\alpha_{i-1},1]$.}
\label{fig:bridging I c}
\end{subfigure}
\end{center}

\caption{Three examples of maps $s$ with a negative right contour point $\alpha_i$, together with the map $\hat{s}_{\alpha_i}$ from Lemma~\ref{lem:bridging I} in red.}
\label{fig:bridging I}
\end{figure}

\begin{proof}
Observe that $[s(\alpha_i),s(\alpha_{i-1})]$ is a liftable range by Lemma~\ref{lem:liftable condition}, and so is $[s(\alpha_i),s(\alpha_{i+1})]$, if $\alpha_i$ is not the last right contour point of $s$, by Lemma~\ref{lem:liftable comparison}.

First suppose that $\alpha_i$ is the last right contour point of $s$ (see Figure~\ref{fig:bridging I c}).  Then $s([\alpha_{i-1},1]) = [s(\alpha_i),s(\alpha_{i-1})]$, and so by Lemma~\ref{lem:stay right} (using $I = [\alpha_{i-1},1]$ and $a = \alpha_{i-1}$), there exists a map $\hat{s}_{\alpha_i} \colon [\alpha_{i-1},1] \to [0,1]$ such that:
\begin{enumerate}[label=(\arabic{*})]
\item $t \circ \hat{s}_{\alpha_i}(x) = f(x)$ for all $x \in [\alpha_{i-1},1]$;
\item $\hat{s}_{\alpha_i}(\alpha_{i-1}) = s(\alpha_{i-1})$;
\item $\hat{s}_{\alpha_i}(x) \geq s(x)$ for all $x \in [\alpha_{i-1},1]$; and
\item $\max \{\hat{s}_{\alpha_i}(x): x \in [\alpha_{i-1},1]$ and $\hat{s}_{\alpha_i}(x) \neq s(x)\} = L(s(\alpha_i))$.
\end{enumerate}
Consider the set $S = \{x \in [\alpha_{i-1},\alpha_i]: s(x) < 0\}$.  Note that $s(S) = [s(\alpha_i),0)$, and $\hat{s}_{\alpha_i}(x) \neq s(x)$ for all $x \in S$ since $\hat{s}_{\alpha_i}(x) \geq 0$ for all $x \in [\alpha_{i-1},1]$.  Let $M = \max \hat{s}_{\alpha_i}(\overline{S})$.  Then $M \leq L(s(\alpha_i))$ by condition (4) above, but also $t([0,M]) \supseteq t \circ \hat{s}_{\alpha_i}(\overline{S}) = t \circ s(\overline{S}) = t([s(\alpha_i),0])$, hence $M \geq L(s(\alpha_i))$ by minimality of $L(s(\alpha_i))$.  Thus $M = L(s(\alpha_i))$, meaning there exists $x \in [\alpha_{i-1},\alpha_i]$ such that $\hat{s}_{\alpha_i}(x) = L(s(\alpha_i))$.

\medskip
Now suppose that $\alpha_i$ is not the last right contour point of $s$ (see Figures~\ref{fig:bridging I a} and \ref{fig:bridging I b}).  In this case we define the map $\hat{s}_{\alpha_i}$ in two ``halves'', first on $[\alpha_{i-1},\alpha_i]$ and then on $[\alpha_i,\alpha_{i+1}]$.  Further, if $\alpha_i$ is the last negative right contour point of $s$ (as in Figure~\ref{fig:bridging I b}), then we go on to define $\hat{s}_{\alpha_i}$ on $[\alpha_{i+1},1]$.

For the left ``half'' $[\alpha_{i-1},\alpha_i]$, $s([\alpha_{i-1},\alpha_i]) = [s(\alpha_i),s(\alpha_{i-1})]$, and so by Lemma~\ref{lem:stay right} (using $I = [\alpha_{i-1},\alpha_i]$ and $a = \alpha_{i-1}$), there exists a map $\hat{s}_L \colon [\alpha_{i-1},\alpha_i] \to [0,1]$ such that:
\begin{enumerate}[label=(L\arabic{*})]
\item $t \circ \hat{s}_L(x) = f(x)$ for all $x \in [\alpha_{i-1},\alpha_i]$;
\item $\hat{s}_L(\alpha_{i-1}) = s(\alpha_{i-1})$;
\item $\hat{s}_L(x) \geq s(x)$ for all $x \in [\alpha_{i-1},\alpha_i]$;
\item $\max \{\hat{s}_L(x): x \in [\alpha_{i-1},\alpha_i]$ and $\hat{s}_L(x) \neq s(x)\} = L(s(\alpha_i))$;
\item $\hat{s}_L(\alpha_i)$ is the largest value $y \in [0,L(s(\alpha_i))]$ such that $t(y) = f(\alpha_i)$.
\end{enumerate}

For the right ``half'' $[\alpha_i,\alpha_{i+1}]$, $s([\alpha_i,\alpha_{i+1}]) = [s(\alpha_i),s(\alpha_{i+1})]$, and so by Lemma~\ref{lem:stay right} (using $I = [\alpha_i,\alpha_{i+1}]$ and $a = \alpha_{i+1}$), there exists a map $\hat{s}_R \colon [\alpha_i,\alpha_{i+1}] \to [0,1]$ such that:
\begin{enumerate}[label=(R\arabic{*})]
\item $t \circ \hat{s}_R(x) = f(x)$ for all $x \in [\alpha_i,\alpha_{i+1}]$;
\item $\hat{s}_R(\alpha_{i+1}) = s(\alpha_{i+1})$;
\item $\hat{s}_R(x) \geq s(x)$ for all $x \in [\alpha_i,\alpha_{i+1}]$;
\item $\max \{\hat{s}_R(x): x \in [\alpha_i,\alpha_{i+1}]$ and $\hat{s}_R(x) \neq s(x)\} = L(s(\alpha_i))$; and
\item $\hat{s}_R(\alpha_i)$ is the largest value $y \in [0,L(s(\alpha_i))]$ such that $t(y) = f(\alpha_i)$.
\end{enumerate}

Now by (L5) and (R5), these two ``halves'' $\hat{s}_L$ and $\hat{s}_R$ agree at $\alpha_i$, so we may paste them together to get a map $\hat{s}_{\alpha_i} \colon [\alpha_{i-1},\alpha_{i+1}] \to [0,1]$ with $\hat{s}_{\alpha_i} {\restriction}_{[\alpha_{i-1},\alpha_i]} = \hat{s}_L$ and $\hat{s}_{\alpha_i} {\restriction}_{[\alpha_i,\alpha_{i+1}]} = \hat{s}_R$.

In the case that $\alpha_i$ is the last negative right contour point of $s$ (meaning $i = n-1$, so $\alpha_{i+1} = \alpha_n$ is the last right contour point of $s$), it remains to define $\hat{s}_{\alpha_i}$ on $[\alpha_{i+1},1]$.  Let $y^- = \min \{s(x): x \in [\alpha_{i+1},1]\}$.  If $y^- \geq 0$, then simply let $\hat{s}_{\alpha_i} {\restriction}_{[\alpha_{i+1},1]} = s {\restriction}_{[\alpha_{i+1},1]}$.  Suppose then that $y^- < 0$.  Observe that $s(\alpha_i) \leq y^-$ (since otherwise there would be negative right departures of $s$ after $\alpha_{i+1}$), so $[y^-,s(\alpha_{i+1})]$ is a liftable range by Lemma~\ref{lem:liftable comparison}, and also $L(y^-) \leq L(s(\alpha_i))$.  Now $s([\alpha_{i+1},1]) = [y^-,s(\alpha_{i+1})]$, so by Lemma~\ref{lem:stay right} (using $I = [\alpha_{i+1},1]$ and $a = \alpha_{i+1}$), there exists a map $\hat{s}_E \colon [\alpha_{i+1},1] \to [0,1]$ such that:
\begin{enumerate}[label=(E\arabic{*})]
\item $t \circ \hat{s}_E(x) = f(x)$ for all $x \in [\alpha_{i+1},1]$;
\item $\hat{s}_E(\alpha_{i+1}) = s(\alpha_{i+1})$;
\item $\hat{s}_E(x) \geq s(x)$ for all $x \in [\alpha_{i+1},1]$; and
\item $\max \{\hat{s}_E(x): x \in [\alpha_{i+1},1]$ and $\hat{s}_E(x) \neq s(x)\} = L(y^-)$.
\end{enumerate}
Then define $\hat{s}_{\alpha_i} {\restriction}_{[\alpha_{i+1},1]} = \hat{s}_E$.

The properties (L1)--(L4) and (R1)--(R4) (and (E1)--(E4) in the case that $\alpha_i$ is the last negative right contour point of $s$) taken together establish the conditions (1)--(4) of the Lemma.
\end{proof}

\begin{lem}[Bridging in II]
\label{lem:bridging II}
Let $f,s \colon [-1,1] \to [-1,1]$ be maps with $f(0) = s(0) = 0$, and suppose $f = t \circ s$.  Let $\beta_j$ be a negative left contour point of $s$, so that $\beta_{j-1}$ is a positive left contour point of $s$ or $\beta_{j-1} = 0$.  Suppose there exists $1 \leq i \leq n$ such that
\begin{enumerate}[label=($\dagger$), ref=($\dagger$)]
\item \label{Y} $s(\alpha_i) \leq s(\beta_j)$ and $s(\alpha_{i-1}) \leq s(\beta_{j-1})$.
\end{enumerate}
Then there exists $\hat{s}_{\beta_j} \colon [\beta_{j+1},\beta_{j-1}] \to [0,1]$ (or $\hat{s}_{\beta_j} \colon [-1,\beta_{j-1}] \to [0,1]$ if $\beta_j$ is the last negative left contour point of $s$) such that:
\begin{enumerate}[label=(\arabic{*})]
\item $t \circ \hat{s}_{\beta_j}(x) = f(x)$ for all $x \in \operatorname{dom} \hat{s}_{\beta_j}$;
\item $\hat{s}_{\beta_j}(\beta_{j-1}) = s(\beta_{j-1})$ and, if $\beta_j$ is not the last left contour point of $s$, then $\hat{s}_{\beta_j}(\beta_{j+1}) = s(\beta_{j+1})$;
\item $\hat{s}_{\beta_j}(x) \geq s(x)$ for all $x \in \operatorname{dom} \hat{s}_{\beta_j}$; and
\item $\max \{\hat{s}_{\beta_j}(x): \hat{s}_{\beta_j}(x) \neq s(x)\} = L(s(\beta_j))$, and there exists $x \in [\beta_j,\beta_{j-1}]$ such that $\hat{s}_{\beta_j}(x) = L(s(\beta_j))$ (recall that the function $L$ is defined in Definition~\ref{defn:L}).
\end{enumerate}
\end{lem}

\begin{proof}
By Lemma~\ref{lem:liftable condition}, $[s(\alpha_i),s(\alpha_{i-1})]$ is a liftable range.  Therefore, in light of the condition \ref{Y}, the same goes for $[s(\beta_j),s(\beta_{j-1})]$ by Lemma~\ref{lem:liftable comparison}.  The remainder of the proof is very similar to that of Lemma~\ref{lem:bridging I}, and we omit the details.
\end{proof}

\section{Common contour factors}
\label{sec:same contour}

The next Lemma is a stengthening of Lemma~7.2 of \cite{ammerlaan-anusic-hoehn2023}, which has the same form, except that, there, an additional assumption was made on the functions $f_1$ and $f_2$, namely that they have ``no contour twins''.  By removing this assumption, we are able to improve on the main result of that paper.

\begin{lem}
\label{lem:bridged s}
Let $f_1,f_2,f_3 \colon [-1,1] \to [-1,1]$ be piecewise-linear maps with $f_1(0) = f_2(0) = f_3(0) = 0$, and such that $t_{f_1} = t_{f_1 \circ f_2}$ and $t_{f_2} = t_{f_2 \circ f_3}$.  Then there exists $\tilde{s} \colon [-1,1] \to [-1,1]$ with $\tilde{s}(0) = 0$ such that:
\begin{enumerate}
\item $t_{f_1} \circ \tilde{s} = f_1 \circ f_2$; and
\item $\tilde{s} \circ t_{f_3}$ has no negative radial departures.
\end{enumerate}
\end{lem}

See Example~\ref{ex:bridged s} following the proof of Lemma~\ref{lem:bridged s} for an illustration of three maps $f_1$,$f_2$,$f_3$ satisfying the conditions of this Lemma, together with the map $\tilde{s}$ (depicted in red in Figure \ref{fig:bridging ex}).

\begin{proof}
For $k = 1,2,3$, let $t_k = t_{f_k}$.  Let $s_1$ be a radial meandering factor of $f_1$ (so that $f_1 = t_1 \circ s_1$).

To construct $\tilde{s}$, we will selectively apply Lemmas~\ref{lem:bridging I} and \ref{lem:bridging II} to the map $s = s_1 \circ f_2$ at certain contour points.  Let $0 < \alpha_1 < \alpha_2 < \ldots < \alpha_n \leq 1$ enumerate all the right contour points of $s_1 \circ f_2$, and let $0 > \beta_1 > \beta_2 > \ldots > \beta_m \geq -1$ enumerate all the left contour points of $s_1 \circ f_2$.  Also, denote $\alpha_0 = \beta_0 = 0$.

Define $B_1 \subseteq \{1,\ldots,n\}$ by the condition $i \in B_1$ if and only if $\alpha_i$ is a negative right contour point of $s_1 \circ f_2$ and $\langle x,\alpha_i \rangle$ is a negative radial departure of $s_1 \circ f_2$ for some $x \in [-1,0)$.  Define $\tilde{s}$ on $[0,1]$ by:
\[ \tilde{s}(x) = \begin{cases}
\hat{s}_{\alpha_i}(x) &\textrm{if $x \in \operatorname{dom} \hat{s}_{\alpha_i}$ for some $i \in B_1$} \\
s_1 \circ f_2(x) &\textrm{otherwise}
\end{cases} \]
where $\hat{s}_{\alpha_i}$ is the map from Lemma~\ref{lem:bridging I}.

Define $B_2 \subseteq \{1,\ldots,m\}$ by the condition $j \in B_2$ if and only if $\beta_j$ is a negative left contour point of $s_1 \circ f_2$ and there exist $x_1^{(j)},x_2^{(j)} \in [-1,1]$ with $x_1^{(j)} < 0 < x_2^{(j)}$ such that:
\begin{enumerate}[label=(\alph{*})]
\item there exists a negative radial departure $\langle w_1,w_2 \rangle$ of $t_3$ such that $t_3(w_1) = x_2^{(j)}$ and $t_3(w_2) = x_1^{(j)}$;
\item $x_2^{(j)}$ is a positive right departure of $\tilde{s}$;
\item $\tilde{s} \left( x_2^{(j)} \right) > \max \left\{ s_1 \circ f_2(x): x \in \left[ x_1^{(j)},0 \right] \right\}$;
\item $x_1^{(j)} \in [\beta_j,\beta_{j-1})$; and
\item $s_1 \circ f_2 \left( x_1^{(j)} \right) < \min \left\{ \tilde{s}(x): x \in \left[ 0,x_2^{(j)} \right] \right\}$.
\end{enumerate}

We prove in Claim~\ref{claim:B2 partner}(2) below that the conditions of Lemma~\ref{lem:bridging II} are met for each $j \in B_2$.  Taking this for granted momentarily, we complete the definition of $\tilde{s}$.  Define $\tilde{s}$ on $[-1,0]$ by:
\[ \tilde{s}(x) = \begin{cases}
\hat{s}_{\beta_j}(x) &\textrm{if $x \in \operatorname{dom} \hat{s}_{\beta_j}$ for some $j \in B_2$} \\
s_1 \circ f_2(x) &\textrm{otherwise}
\end{cases} \]
where $\hat{s}_{\beta_j}$ is the map from Lemma~\ref{lem:bridging II}.

The reader may wish to pause here and review Example~\ref{ex:bridged s} which follows the proof of Lemma~\ref{lem:bridged s}, to see the construction of this map $\tilde{s}$ carried out for some particular maps $f_1,f_2,f_3$.

\medskip

Note that by construction, $t_1 \circ \tilde{s} = f_1 \circ f_2$, and also
\[ \tilde{s}(x) \geq s_1 \circ f_2(x) \textrm{ for all $x \in [-1,1]$.} \]

The next Claim gives some information about the locations of right and left departures of $\tilde{s}$.  The first two parts, about right departures, will be used in the proof of Claim~\ref{claim:B2 partner}.  The last two parts, about left departures, involve the map $\tilde{s} {\restriction}_{[-1,0]}$ which, as mentioned above, should technically be considered only after Claim~\ref{claim:B2 partner}(2) is established.  However, because these parts are so similar, we state them together here.

\begin{claim}
\label{claim:s tilde deps}
\begin{enumerate}
\item If $x \in (0,1]$ is a positive right departure of $\tilde{s}$, then either $x$ is a positive right departure of $s_1 \circ f_2$ and $\tilde{s}(x) = s_1 \circ f_2(x)$, or $x \in (\alpha_{i-1},\alpha_i]$ for some $i \in B_1$.
\item If $x \in (0,1]$ is a negative right departure of $\tilde{s}$, then there exists $i \in \{1,\ldots,n\}$ such that $\alpha_i$ is a negative right contour point of $s_1 \circ f_2$, $i \notin B_1$, and $x \in (\alpha_{i-1},\alpha_i]$.
\item If $x \in [-1,0)$ is a positive left departure of $\tilde{s}$, then either $x$ is a positive left departure of $s_1 \circ f_2$ and $\tilde{s}(x) = s_1 \circ f_2(x)$, or $x \in [\beta_j,\beta_{j-1})$ for some $j \in B_2$.
\item If $x \in [-1,0)$ is a negative left departure of $\tilde{s}$, then there exists $j \in \{1,\ldots,m\}$ such that $\beta_j$ is a negative left contour point of $s_1 \circ f_2$, $j \notin B_2$, and $x \in [\beta_j,\beta_{j-1})$.
\end{enumerate}
\end{claim}

\begin{proof}[Proof of Claim~\ref{claim:s tilde deps}]
\renewcommand{\qedsymbol}{\textsquare (Claim~\ref{claim:s tilde deps})}
For (1), let $x \in (0,1]$ be a positive right departure of $\tilde{s}$.  If $\tilde{s}(x) = s_1 \circ f_2(x)$, then since $\tilde{s}(x') \geq s_1 \circ f_2(x')$ for all $x' \in [-1,1]$, it follows that $x$ is a positive right departure of $s_1 \circ f_2$ as well.  Suppose then that $\tilde{s}(x) \neq s_1 \circ f_2(x)$.  This means that $x \in \operatorname{dom} \hat{s}_{\alpha_i}$ for some $i \in B_1$.  By property (4) of Lemma~\ref{lem:bridging I}, the maximum value of $\hat{s}_{\alpha_i}(x')$ where $x' \in \operatorname{dom} \hat{s}_{\alpha_i}$ and $\hat{s}_{\alpha_i}(x') \neq s_1 \circ f_2(x')$ is achieved on the interval $[\alpha_{i-1},\alpha_i]$, which means that we must have $x \in (\alpha_{i-1},\alpha_i]$ (note that $x \neq \alpha_{i-1}$ because $\tilde{s}(\alpha_{i-1}) = s_1 \circ f_2(\alpha_{i-1})$).

For (2), let $x \in (0,1]$ be a negative right departure of $\tilde{s}$.  Let $i \in \{1,\ldots,n\}$ be such that $\alpha_i$ is a negative right contour point of $s_1 \circ f_2$, and either $i < n-1$ and $x \in [\alpha_{i-1},\alpha_{i+1}]$, or $i \geq n-1$ (i.e.\ $\alpha_i$ is the last negative right contour point of $s_1 \circ f_2$) and $x \in [\alpha_{i-1},1]$.  If it were the case that $i \in B_1$, then by construction, $x \in \operatorname{dom} \hat{s}_{\alpha_i}$, so $\tilde{s}(x) = \hat{s}_{\alpha_i}(x) \geq 0$, a contradiction since $x$ is a negative right departure of $\tilde{s}$.  Hence, $i \notin B_1$.  This means $\tilde{s}(x') = s_1 \circ f_2(x')$ for each $x' \in [\alpha_{i-1},\alpha_{i+1}]$ (or $[\alpha_{i-1},1]$ in the case that $\alpha_i$ is the last negative right contour point of $s_1 \circ f_2$), implying that $x \in (\alpha_{i-1},\alpha_i]$ since $x$ is a negative right departure of $\tilde{s}$.

The proofs of (3) and (4) are similar.
\end{proof}

We will make use of the following property several times in the proof below:
\begin{enumerate}[label=($\ast$), ref=($\ast$)]
\item \label{properly nested} If $\langle z_1,z_2 \rangle$ is a radial departure of $s_1 \circ f_2$, then for any $j \in B_2$, either $z_1 < x_1^{(j)} < 0 < x_2^{(j)} < z_2$ or $x_1^{(j)} < z_1 < 0 < z_2 < x_2^{(j)}$.
\end{enumerate}
This can be seen as follows.  It follows from Lemma~3.6 of \cite{ammerlaan-anusic-hoehn2023} that $f_2 \circ f_3$ and $f_2 \circ t_3$ have the same radial contour factor, so $f_2$ and $f_2 \circ t_3$ have the same radial contour factor.  Property \ref{properly nested} then follows from Proposition~\ref{prop:comp same dep} since any $\langle z_1,z_2 \rangle$ as in \ref{properly nested} is a radial departure of $f_2$ by Proposition~\ref{prop:comp dep}.

\begin{claim}
\label{claim:B2 partner}
For each $j \in B_2$:
\begin{enumerate}
\item There exists $k \in B_1$ such that $x_2^{(j)} \in (\alpha_{k-1},\alpha_{k+1}]$ (or $(\alpha_{k-1},1]$ in the case $k = n$) and $s_1 \circ f_2(\alpha_k) \leq s_1 \circ f_2(\beta_j)$.
\item There exists $k' \in B_1$ such that $k' \leq k$, $s_1 \circ f_2(\alpha_{k'}) \leq s_1 \circ f_2(\beta_j)$, and $s_1 \circ f_2(\alpha_{k'-1}) \leq s_1 \circ f_2(\beta_{j-1})$.
\end{enumerate}
\end{claim}

\begin{proof}[Proof of Claim~\ref{claim:B2 partner}]
\renewcommand{\qedsymbol}{\textsquare (Claim~\ref{claim:B2 partner})}
Let $j \in B_2$, and for simplicity denote $x_1 = x_1^{(j)}$ and $x_2 = x_2^{(j)}$.

First we point out that it cannot be the case that $\alpha_1$ is a positive right contour point of $s_1 \circ f_2$ and $x_2 \in [0,\alpha_1]$.  This is because if this were the case, then taking any negative left departure $x \in [x_1,0]$ of $s_1 \circ f_2$, we would have that $\langle x,x_2 \rangle$ is a (positive) radial departure of $s_1 \circ f_2$, contradicting \ref{properly nested} (with $z_1 = x$, $z_2 = x_2$).  Also, it cannot be the case that $\alpha_n$ is a positive right contour point of $s_1 \circ f_2$ and $x_2 \in (\alpha_n,1]$, since such a value $x_2$ cannot be a right departure of $\tilde{s}$, according to Claim~\ref{claim:s tilde deps}(1).

Choose $k \in \{1,\ldots,n\}$ such that $\alpha_k$ is a negative right contour point of $s_1 \circ f_2$ and $x_2 \in (\alpha_{k-1},\alpha_{k+1}]$ (or $x_2 \in (\alpha_{k-1},1]$ in the case that $k = n$).  We claim that $k \in B_1$.  Indeed, if this were not the case, then by Claim~\ref{claim:s tilde deps}(1), $x_2$ would be a positive right departure of $s_1 \circ f_2$ and $\tilde{s}(x_2) = s_1 \circ f_2(x_2)$.  Taking the negative left departure $x_1' \in [x_1,0]$ of $s_1 \circ f_2$ with $s_1 \circ f_2(x_1') = s_1 \circ f_2(x_1)$, we would have that $\langle x_1',x_2 \rangle$ is a (positive) radial departure of $s_1 \circ f_2$ by the conditions (c) and (e) from the definition of $B_2$, contradicting \ref{properly nested} (with $z_1 = x_1'$, $z_2 = x_2$).

Since $k \in B_1$, there exists $x \in [-1,0)$ such that $\langle x,\alpha_k \rangle$ is a negative radial departure of $s_1 \circ f_2$.  Suppose for a contradiction that $s_1 \circ f_2(\alpha_k) > s_1 \circ f_2(\beta_j)$.  Then we must have $x \geq \beta_{j-1} > x_1$, and so by \ref{properly nested} (with $z_1 = x$, $z_2 = \alpha_k$) we deduce that $x_2 > \alpha_k$.  By Claim~\ref{claim:s tilde deps}(1), this means that $x_2$ is a positive right departure of $s_1 \circ f_2$ and $\tilde{s}(x_2) = s_1 \circ f_2(x_2)$.  But then we claim that $\langle \beta_j,x_2 \rangle$ is a (positive) radial departure of $s_1 \circ f_2$.  Indeed, for any $p \in [0,x_2)$, $s_1 \circ f_2(p) < s_1 \circ f_2(x_2)$ since $x_2$ is a positive right departure of $s_1 \circ f_2$ and
\begin{align*}
s_1 \circ f_2(p) &\geq s_1 \circ f_2(\alpha_k) && \textrm{since $\alpha_k$ is the last negative right contour} \\
& && \textrm{point of $s_1 \circ f_2$ before $x_2$} \\
&> s_1 \circ f_2(\beta_j) ,
\end{align*}
and for any $p \in (\beta_j,0]$,
\begin{align*}
s_1 \circ f_2(p) &\leq s_1 \circ f_2(\beta_{j-1}) && \textrm{since $\beta_{j-1}$ is the last (moving to the left from $0$)} \\
& && \textrm{positive left contour point of $s_1 \circ f_2$ before $\beta_j$} \\
&< \tilde{s}(x_2) && \textrm{by condition (c) from the definition of $B_2$} \\
&= s_1 \circ f_2(x_2) ,
\end{align*}
and $s_1 \circ f_2(p) > s_1 \circ f_2(\beta_j)$ since $\beta_j$ is a negative left departure of $s_1 \circ f_2$.  This is a contradiction by \ref{properly nested} (with $z_1 = \beta_j$, $z_2 = x_2$).  Therefore, $s_1 \circ f_2(\alpha_k) \leq s_1 \circ f_2(\beta_j)$.  This proves (1).

For (2), let $k' \in \{1,\ldots,n\}$ be minimal such that $s_1 \circ f_2(\alpha_{k'}) \leq s_1 \circ f_2(\beta_j)$.  Clearly $k'$ exists and $k' \leq k$.  Also, $k' \in B_1$ since either $k' = k \in B_1$, or $k' < k$, in which case $\alpha_{k'} < x_2$ and so $\tilde{s}(\alpha_{k'}) > s_1 \circ f_2(x_1) \geq s_1 \circ f_2(\beta_j)$, meaning $\tilde{s}(\alpha_{k'}) \neq s_1 \circ f_2(\alpha_{k'})$.  If it were the case that $s_1 \circ f_2(\alpha_{k'-1}) > s_1 \circ f_2(\beta_{j-1})$, then $\langle \beta_j,\alpha_{k'-1} \rangle$ would be a (positive) radial departure of $s_1 \circ f_2$, contradicting \ref{properly nested} (with $z_1 = \beta_j$, $z_2 = \alpha_{k'-1}$).  Therefore, $s_1 \circ f_2(\alpha_{k'-1}) \leq s_1 \circ f_2(\beta_{j-1})$, as desired.
\end{proof}

To complete the proof of Lemma~\ref{lem:bridged s}, it remains to prove that $\tilde{s} \circ t_3$ has no negative radial departures.  According to Proposition~\ref{prop:comp dep}, any negative radial departure of $\tilde{s} \circ t_3$ must come from a negative radial departure of $t_3$ mapping onto a positive radial departure of $\tilde{s}$, or from a positive radial departure of $t_3$ mapping onto a negative radial departure of $\tilde{s}$.  We rule out these two possibilities in Claims~\ref{claim:neg pos} and \ref{claim:pos neg} below.

\begin{claim}
\label{claim:neg pos}
For any negative radial departure $\langle w_1,w_2 \rangle$ of $t_3$, $\langle t_3(w_2),t_3(w_1) \rangle$ is not a positive radial departure of $\tilde{s}$.
\end{claim}

\begin{proof}[Proof of Claim~\ref{claim:neg pos}]
\renewcommand{\qedsymbol}{\textsquare (Claim~\ref{claim:neg pos})}
Let $\langle w_1,w_2 \rangle$ be a negative radial departure of $t_3$, let $x_1 = t_3(w_2)$ and $x_2 = t_3(w_1)$, and suppose for a contradiction that $\langle x_1,x_2 \rangle$ is a positive radial departure of $\tilde{s}$.

By Claim~\ref{claim:s tilde deps}(4), there exists $j \in \{1,\ldots,m\}$ such that $\beta_j$ is a negative left contour point of $s_1 \circ f_2$, $j \notin B_2$, and $x_1 \in [\beta_j,\beta_{j-1})$.  Furthermore, $x_2$ is a positive right departure of $\tilde{s}$, $\tilde{s}(x_2) > \max \{\tilde{s}(x): x \in [x_1,0]\} \geq \max \{s_1 \circ f_2(x): x \in [x_1,0]\}$, and $\tilde{s}(x_1) = s_1 \circ f_2(x_1) < \min \{\tilde{s}(x): x \in [0,x_2]\}$.  Thus the conditions (a)--(e) of the definition of $B_2$ are satisfied for $x_1^{(j)} = x_1$ and $x_2^{(2)} = x_2$, so we conclude that $j \in B_2$, a contradiction.
\end{proof}

In the next Claim we will make use of the following property, which follows from Proposition~\ref{prop:alt dep nested} (see the comment immediately following Proposition~\ref{prop:alt dep nested}).
\begin{enumerate}[label=($\ast\ast$), ref=($\ast\ast$)]
\item \label{properly nested 2} If $\langle v_1,v_2 \rangle$ is a positive radial departure of $t_3$, then for any $j \in B_2$, either $t_3(v_1) < x_1^{(j)} < 0 < x_2^{(j)} < t_3(v_2)$ or $x_1^{(j)} < t_3(v_1) < 0 < t_3(v_2) < x_2^{(j)}$.
\end{enumerate}

\begin{claim}
\label{claim:pos neg}
For any positive radial departure $\langle v_1,v_2 \rangle$ of $t_3$, $\langle t_3(v_1),t_3(v_2) \rangle$ is not a negative radial departure of $\tilde{s}$.
\end{claim}

\begin{proof}[Proof of Claim~\ref{claim:pos neg}]
\renewcommand{\qedsymbol}{\textsquare (Claim~\ref{claim:pos neg})}
Let $\langle v_1,v_2 \rangle$ be a positive radial departure of $t_3$, let $x_1 = t_3(v_1)$ and $x_2 = t_3(v_2)$, and suppose for a contradiction that $\langle x_1,x_2 \rangle$ is a negative radial departure of $\tilde{s}$.  By Claim~\ref{claim:s tilde deps}(2), there exists $i \in \{1,\ldots,n\}$ such that $\alpha_i$ is a negative right contour point of $s_1 \circ f_2$, $i \notin B_1$, and $x_2 \in (\alpha_{i-1},\alpha_i]$.

\medskip
\noindent
\textbf{Case~1:} Suppose $\tilde{s}(x_1) = s_1 \circ f_2(x_1)$.

Then $s_1 \circ f_2(x_1) > s_1 \circ f_2(x)$ for all $x \in (x_1,x_2]$.  Since $i \notin B_1$, we know that $\langle x_1,\alpha_i \rangle$ is not a negative radial departure of $s_1 \circ f_2$, meaning that there exists a negative left contour point $\beta_j \in (x_1,0)$ such that $s_1 \circ f_2(\beta_j) \leq s_1 \circ f_2(\alpha_i)$.  Note that $j \in B_2$, since otherwise we would have $\tilde{s}(\beta_j) = s_1 \circ f_2(\beta_j) \leq s_1 \circ f_2(\alpha_i) \leq s_1 \circ f_2(x_2) = \tilde{s}(x_2)$, contradicting the assumption that $\langle x_1,x_2 \rangle$ is a negative radial departure of $\tilde{s}$.  By Claim~\ref{claim:B2 partner}(1), there exists $k \in B_1$ such that $\alpha_{k-1} < x_2^{(j)}$ and $s_1 \circ f_2(\alpha_k) \leq s_1 \circ f_2(\beta_j)$.  Since $s_1 \circ f_2(\alpha_k) \leq s_1 \circ f_2(\beta_j) \leq s_1 \circ f_2(\alpha_i)$, and $i \neq k$ since $i \notin B_1$ but $k \in B_1$, we conclude that $i < k$.  Now $x_1 < \beta_j \leq x_1^{(j)}$ and $x_2 \leq \alpha_i < \alpha_{k-1} < x_2^{(j)}$, contradicting \ref{properly nested 2}.

\medskip
\noindent
\textbf{Case~2:} Suppose $\tilde{s}(x_1) \neq s_1 \circ f_2(x_1)$.

This means that $x_1 \in [\beta_j,\beta_{j-1})$ for some $j \in B_2$, by Claim~\ref{claim:s tilde deps}(3).  Also, $\tilde{s}(x_1) \leq L(s_1 \circ f_2(\beta_j))$ by Lemma~\ref{lem:bridging II}(4).

We claim that $s_1 \circ f_2(\alpha_{i-1}) < s_1 \circ f_2(\beta_{j-1})$.  Indeed, suppose for a contradiction that $s_1 \circ f_2(\alpha_{i-1}) \geq s_1 \circ f_2(\beta_{j-1})$.  According to Claim~\ref{claim:B2 partner}(2), there exists $k' \in B_1$ such that $s_1 \circ f_2(\alpha_{k'}) \leq s_1 \circ f_2(\beta_j)$, and $s_1 \circ f_2(\alpha_{k'-1}) \leq s_1 \circ f_2(\beta_{j-1})$.  Observe that $k' \leq i$ since $s_1 \circ f_2(\alpha_{k'-1}) \leq s_1 \circ f_2(\beta_{j-1}) \leq s_1 \circ f_2(\alpha_{i-1})$, and in fact $k' < i$ since $k' \in B_1$ and $i \notin B_1$.  By Lemma~\ref{lem:bridging I}(4), there exists $x' \in [\alpha_{k'-1},\alpha_{k'}]$ such that
\begin{align*}
\tilde{s}(x') &= L(s_1 \circ f_2(\alpha_{k'})) \\
& \geq L(s_1 \circ f_2(\beta_j)) \quad \textrm{by monotonicity of $L$} \\
&\geq \tilde{s}(x_1) .
\end{align*}
But this is a contradiction since $\langle x_1,x_2 \rangle$ is a radial departure of $\tilde{s}$.

We now derive a contradiction in Case~2 by mimicking the argument in Case~1, with $\beta_{j-1}$ playing the role of $x_1$.  We have $s_1 \circ f_2(\beta_{j-1}) > s_1 \circ f_2(x)$ for all $x \in (\beta_{j-1},0]$ since $\beta_{j-1}$ is a positive left departure of $s_1 \circ f_2$, and $s_1 \circ f_2(\beta_{j-1}) > s_1 \circ f_2(x)$ for all $x \in [0,\alpha_i]$ since $\alpha_{i-1}$ is a positive right contour point of $s_1 \circ f_2$ and $s_1 \circ f_2(\beta_{j-1}) > s_1 \circ f_2(\alpha_{i-1})$.  Thus $s_1 \circ f_2(\beta_{j-1}) > s_1 \circ f_2(x)$ for all $x \in (\beta_{j-1},\alpha_i]$.  Since $i \notin B_1$, we know that $\langle \beta_{j-1},\alpha_i \rangle$ is not a negative radial departure of $s_1 \circ f_2$, meaning that there exists $j' \in B_2$ such that $j' < j$ and $s_1 \circ f_2(\beta_{j'}) \leq s_1 \circ f_2(\alpha_i)$.  By Claim~\ref{claim:B2 partner}(1) applied to $j'$, there exists $k \in B_1$ such that $\alpha_{k-1} < x_2^{(j')}$ and $s_1 \circ f_2(\alpha_k) \leq s_1 \circ f_2(\beta_{j'})$.  Since $s_1 \circ f_2(\alpha_k) \leq s_1 \circ f_2(\beta_{j'}) \leq s_1 \circ f_2(\alpha_i)$, and $i \neq k$ since $i \notin B_1$ but $k \in B_1$, we conclude that $i < k$.  Now $x_1 < \beta_{j-1} < \beta_{j'} \leq x_1^{(j')}$ and $x_2 \leq \alpha_i < \alpha_{k-1} < x_2^{(j')}$, contradicting \ref{properly nested 2}.
\end{proof}

This completes the proof of Lemma~\ref{lem:bridged s}.
\end{proof}

\begin{example}
\label{ex:bridged s}
We provide an example of three maps $f_1$,$f_2$,$f_3$ satisfying the conditions of Lemma~\ref{lem:bridged s} in Figure~\ref{fig:bridging ex}.  In this example, $t_{f_1} = f_1$ and $t_{f_3} = f_3$.  This means in particular that a radial meandering factor for $f_1$ is $s_1 = \mathrm{id}$, so that $s_1 \circ f_2 = f_2$.

\begin{figure}
\begin{center}
\includegraphics{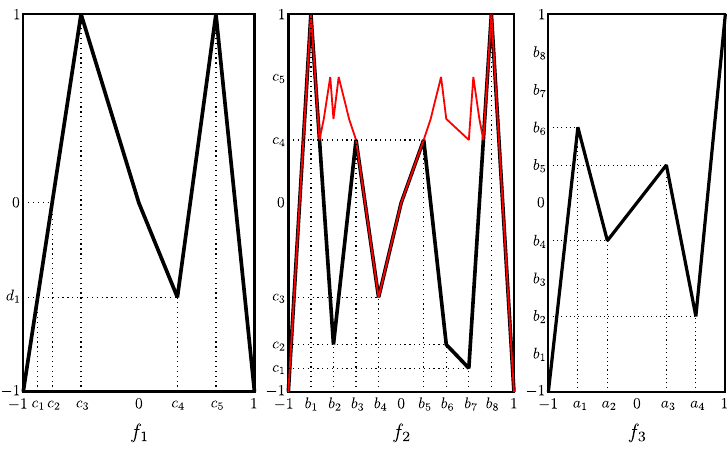}
\end{center}

\caption{An example of three maps $f_1$,$f_2$,$f_3$ satisfying the conditions of Lemma~\ref{lem:bridged s}, together with the map $\tilde{s}$ (in red) as constructed in the proof of Lemma~\ref{lem:bridged s}.}
\label{fig:bridging ex}
\end{figure}

We illustrate the construction of the map $\tilde{s}$ from Lemma~\ref{lem:bridged s} for this example.  The right contour points of $s_1 \circ f_2 = f_2$ are $\alpha_1 = b_5$, $\alpha_2 = b_7$, $\alpha_3 = b_8$, and $\alpha_4 = 1$.  Of these, $b_7$ and $1$ are the negative right contour points.  Note that $\alpha_2 = b_7 \in B_1$, because, for example, $\langle b_1,b_7 \rangle$ is a negative radial departure of $s_1\circ f_2=f_2$.  However, $\alpha_4 = 1 \notin B_1$, because there is no $x \in [-1,0)$ such that $\langle x,1 \rangle$ is a negative radial departure of $s_1\circ f_2=f_2$.

Therefore, the map $\tilde{s} {\restriction}_{[0,1]}$ for this example is given by
\[ \tilde{s}(x) = \begin{cases}
\hat{s}_{\alpha_2}(x) &\textrm{if $x \in [\alpha_1,\alpha_3]$} \\
s_1 \circ f_2(x) &\textrm{otherwise.}
\end{cases} \]

The left contour points of $s_1 \circ f_2 = f_2$ are $\beta_1 = b_4$, $\beta_2 = b_3$, $\beta_3 = b_2$, $\beta_4 = b_1$, and $\beta_5 = -1$.  Of these, $b_4$, $b_2$, and $-1$ are the negative left contour points.  We claim that $\beta_3 = b_2 \in B_2$.  To see this, let $x_1^{(3)} = b_2$, and let $x_2^{(3)} \in (b_5,b_6)$ be such that $f_2 \left( x_2^{(3)} \right) = c_3$.  Then let $w_2 = a_4$, and let $w_1 \in (a_1,a_2)$ be such that $f_3(w_1) = x_2^{(3)}$.  Since:
\begin{enumerate}[label=(\alph{*})]
\item $\langle w_1,w_2 \rangle$ is a negative radial departure of $f_3 = t_3$,
\item $x_2^{(3)}$ is a positive right departure of $\tilde{s}$,
\item $\tilde{s} \left( x_2^{(3)} \right) = c_5 > c_4 = \max \left\{ f_2(x): x \in \left[ x_1^{(3)},0 \right] \right\}$,
\item $x_1^{(3)} = b_2 \in [b_2,b_3) = [\beta_3,\beta_2)$, and
\item $f_2 \left( x_1^{(3)} \right) = c_2 < 0 = \min \left\{ \tilde{s}(x): x \in \left[ 0,x_2^{(3)} \right] \right\}$,
\end{enumerate}
it follows that $\beta_3 = b_2 \in B_2$.

On the other hand, we claim that $\beta_1 = b_4 \notin B_2$ and $\beta_5 = -1 \notin B_2$.  To see this, note that if $\langle w_1,w_2 \rangle$ is a negative radial departure of $f_3 = t_3$, then $w_1 \in [a_1,a_2)$, $f_3(w_1) \in (b_5,b_6]$, $w_2 \in (a_3,a_4]$, and $f_3(w_2) \in [b_2,b_4)$.  In particular, we have that $f_3(w_2) \notin [b_4,0) = [\beta_1,\beta_0)$ and $f_3(w_2) \notin [-1,b_1) = [\beta_5,\beta_4)$, and hence condition (d) of the definition of $B_2$ cannot be met for $\beta_1 = b_4$ and $\beta_5 = -1$.

Therefore, the map $\tilde{s} {\restriction}_{[-1,0]}$ for this example is given by
\[ \tilde{s}(x) = \begin{cases}
\hat{s}_{\beta_3}(x) &\textrm{if $x \in [\beta_4,\beta_2]$} \\
s_1 \circ f_2(x) &\textrm{otherwise.}
\end{cases} \]

The reader may find it instructive to verify directly in this example that the composition $\tilde{s} \circ t_{f_3} = \tilde{s} \circ f_3$ has no negative radial departures.
\end{example}

\begin{thm}
\label{thm:same contour}
Let $f_n \colon [-1,1] \to [-1,1]$, $n = 1,2,\ldots$, be piecewise-linear maps with $f_n(0) = 0$ for each $n$.  Suppose that for each $n$, $f_n$ and $f_n \circ f_{n+1}$ have the same radial contour factor.  Then there exists an embedding of $X = \varprojlim \left \langle [-1,1], f_n \right \rangle$ into $\mathbb{R}^2$ for which the point $\langle 0,0,\ldots \rangle \in X$ is accessible.
\end{thm}

The proof of Theorem~\ref{thm:same contour} is identical to the proof of Theorem~7.3 in \cite{ammerlaan-anusic-hoehn2023}, except that we substitute Lemma~\ref{lem:bridged s} above for Lemma~7.2 of that paper.  We include the details for completeness.

\begin{proof}
For each odd $n \geq 1$, let $t_n$ be the radial contour factor of $f_n$, and apply Lemma~\ref{lem:bridged s} to obtain a map $\tilde{s}_n \colon [-1,1] \to [-1,1]$ with $\tilde{s}_n(0) = 0$ such that $t_n \circ \tilde{s}_n = f_n \circ f_{n+1}$ and $\tilde{s}_n \circ t_{n+2}$ has no negative radial departures.  Then by Proposition~\ref{prop:embed 0 accessible} there exists an embedding of $\displaystyle \varprojlim \left \langle [-1,1], \tilde{s}_n \circ t_{n+2} \right \rangle_{n \textrm{ odd}}$ into $\mathbb{R}^2$ for which the point $\langle 0,0,\ldots \rangle$ is accessible.

Also, by standard properties of inverse limits,
\begin{align*}
X = \varprojlim \left \langle [-1,1],f_n \right \rangle &\approx \varprojlim \left \langle [-1,1],f_n \circ f_{n+1} \right \rangle_{n \textrm{ odd}} \quad \textrm{by composing bonding maps} \\
&= \varprojlim \left \langle [-1,1],t_n \circ \tilde{s}_n \right \rangle_{n \textrm{ odd}} \\
&\approx \varprojlim \left \langle [-1,1], t_1, [-1,1], \tilde{s}_1, [-1,1], t_3, [-1,1], \tilde{s}_3, \ldots \right \rangle \\ & \hspace{1.5in} \textrm{by (un)composing bonding maps} \\
&\approx \varprojlim \left \langle [-1,1], \tilde{s}_1, [-1,1], t_3, [-1,1], \tilde{s}_3, [-1,1], t_5, \ldots \right \rangle \\ & \hspace{1.5in} \textrm{by dropping first coordinate} \\
&\approx \varprojlim \left \langle [-1,1],\tilde{s}_n \circ t_{n+2} \right \rangle_{n \textrm{ odd}} \quad \textrm{by composing bonding maps.}
\end{align*}
In fact, an explicit homeomorphism between $X$ and $\displaystyle \varprojlim \left \langle [-1,1],\tilde{s}_n \circ t_{n+2} \right \rangle_{n \textrm{ odd}}$ is given by
\[ h \left( \langle x_n \rangle_{n=1}^\infty \right) = \langle \tilde{s}_{2k-1}(x_{2k+1}) \rangle_{k=1}^\infty .\]  Clearly $h(\langle 0,0,\ldots \rangle) = \langle 0,0,\ldots \rangle$.  Thus there exists an embedding of $X$ into $\mathbb{R}^2$ for which the point $\langle 0,0,\ldots \rangle \in X$ is accessible.
\end{proof}

\section{Simplicial inverse systems of arcs}
\label{sec:simplicial}

As an application of Theorem~\ref{thm:same contour}, we give a positive answer to the question of Nadler and Quinn for a broad class of arc-like continua, in Theorem~\ref{thm:simplicial} below.

An inverse system $\left \langle [-1,1],f_n \right \rangle$ is \emph{simplicial} if there exist finite sets $S_1,S_2,\ldots$ such that for each $n$, $-1,1 \in S_n$, $f_n(S_{n+1}) \subseteq S_n$, and if $I$ is any component of $[-1,1] \smallsetminus S_{n+1}$ then $f_n$ is either constant on $I$, or $f_n$ is linear on $I$ and $f_n(I) \cap S_n = \emptyset$.

Examples of arc-like continua which are inverse limits of simplicial inverse systems of arcs include Knaster continua (see e.g.\ \cite{rogers1970}), the two examples described by Minc in \cite[Figure 3, p.297]{problems2018}, and all continua which are inverse limits of a single bonding map that is piecewise monotone and post-critically finite.  As such, the next Theorem generalizes Corollary 5.8 of \cite{anusic2021}.  For other results about accessible points of Knaster continua, see \cite{mayer1982}, \cite{mayer1983}, and \cite{debski-tymchatyn1993}.

\begin{thm}
\label{thm:simplicial}
If $X$ is an arc-like continuum which is the inverse limit of a simplicial inverse system of arcs, and if $x \in X$, then there exists an embedding of $X$ into $\mathbb{R}^2$ for which $x$ is an accessible point.
\end{thm}

\begin{proof}
Let $\left \langle [-1,1],f_n \right \rangle$ be a simplicial inverse system of arcs, and let $S_n$, $n = 1,2,\ldots$, be finite sets as in the definition of a simplicial inverse system.  Let $x = \langle x_1,x_2,\ldots \rangle \in X = \varprojlim \left \langle [-1,1],f_n \right \rangle$.  We first show that we may assume without loss of generality that $x_n = 0$ for all $n$.

It is well-known that if $x_n = \pm 1$ for infinitely many $n$, then $x$ is an endpoint of $X$ and one can easily embed $X$ in $\mathbb{R}^2$ so as to make $x$ an accessible point (see e.g.\ \cite[p.295]{problems2018}).  Hence we may as well assume that $x_n \neq \pm 1$ for all but finitely many $n$, and in fact by dropping finitely many coordinates we may assume that $x_n \neq \pm 1$ for all $n$.

Note that if $S_n = \{-1,1\}$ for all but finitely many $n$, then $X$ is either an arc or a single point, and the result follows immediately.  Therefore, we may assume that $S_n$ has at least one interior point of $[-1,1]$ for infinitely many $n$, and in fact, by composing bonding maps if necessary, we may assume that $S_n$ has at least one interior point of $[-1,1]$ for all $n$.

Now, for each $n = 1,2,\ldots$, let $h_n \colon [-1,1] \to [-1,1]$ be a piecewise-linear homeomorphism such that $h_n(x_n) = 0$, and for each component $I$ of $[-1,1] \smallsetminus S_n$, $h_n$ is linear on $I$.  Define $f_n' = h_n \circ f_n \circ h_{n+1}^{-1}$ and $S_n' = h_n(S_n)$.  Then it is straightforward to see that $\left \langle [-1,1],f_n' \right \rangle$ is a simplicial inverse system of arcs with respect to these finite sets $S_n'$.  Moreover, $\varprojlim \left \langle [-1,1],f_n' \right \rangle \approx X$, and in fact an explicit homeomorphism $h \colon X \to \varprojlim \left \langle [-1,1],f_n' \right \rangle$ is given by $h \left( \langle y_n \rangle_{n=1}^\infty \right) = \langle h_n(y_n) \rangle_{n=1}^\infty$, and clearly $h(x) = \langle 0,0,\ldots \rangle$.  Thus, we may assume that $x_n = 0$ for all $n$.

We point out that we may also assume that for each $n \geq 1$, the restrictions $f_n {\restriction}_{[-1,0]}$ and $f_n {\restriction}_{[0,1]}$ are both non-constant.  This is because if either $f_n {\restriction}_{[-1,0]}$ or $f_n {\restriction}_{[0,1]}$ is constant for infinitely many $n$, then by taking compositions of the bonding maps we may obtain that either $f_n {\restriction}_{[-1,0]}$ or $f_n {\restriction}_{[0,1]}$ is constant for all $n$, but then none of the maps $f_n$ would have any radial departures, hence by Proposition~\ref{prop:embed 0 accessible} there would exist an embedding of $X$ into $\mathbb{R}^2$ for which the point $\langle 0,0,\ldots \rangle \in X$ is accessible.

To complete the proof, we show that we may take compositions of the bonding maps $f_n$ to produce an inverse system for which the conditions of Theorem~\ref{thm:same contour} are met.  For each $n,m$ with $n < m$, define $f^m_n = f_n \circ f_{n+1} \circ \cdots \circ f_{m-1}$.  In this notation, $f_n = f^{n+1}_n$.  Also, if $n < k < m$, then $f^m_n = f^k_n \circ f^m_k$.

Let $n_1 = 1$.  For each $m = 2,3,\ldots$, for any radial contour point $x$ of $f^m_1$, clearly $f^m_1(x) \in S_1$.  Since $S_1$ is finite, it follows that there are only finitely many possible radial contour factors for the maps $f^m_1$, $m = 2,3,\ldots$.  Thus there exists an infinite set of integers $J_1 \subseteq \{2,3,\ldots\}$ such that for each $m,m' \in J_1$, the maps $f^m_1$ and $f^{m'}_1$ have the same radial contour factor.

Continuing recursively, let $k \geq 2$ and suppose $n_{k-1}$ and $J_{k-1}$ have been defined.  Let $n_k$ be the smallest element of $J_{k-1}$.  As above, since $S_{n_k}$ is finite, it follows that there are only finitely many possible radial contour factors for the maps $f^m_{n_k}$, $m \in J_{k-1} \smallsetminus \{n_k\}$.  Thus there exists an infinite set of integers $J_k \subseteq J_{k-1} \smallsetminus \{n_k\}$ such that for each $m,m' \in J_k$, the maps $f^m_{n_k}$ and $f^{m'}_{n_k}$ have the same radial contour factor.

For each $k = 1,2,\ldots$, define $F_k = f^{n_{k+1}}_{n_k}$.  Now we have that $X \approx \varprojlim \left \langle [-1,1], F_k \right \rangle$, since this latter system is obtained from the original by composing bonding maps, and clearly the point $x = \langle 0,0,\ldots \rangle$ in the original system corresponds to the point $\langle 0,0,\ldots \rangle$ in this latter system.  Observe that for each $k = 1,2,\ldots$,
\[ F_k \circ F_{k+1} = f^{n_{k+1}}_{n_k} \circ f^{n_{k+2}}_{n_{k+1}} = f^{n_{k+2}}_{n_k} \]
and this map has the same radial contour factor as $F_k = f^{n_{k+1}}_{n_k}$ by the choice of $J_k$.  Therefore, by Theorem~\ref{thm:same contour}, there exists an embedding of $X \approx \varprojlim \left \langle [-1,1], F_k \right \rangle$ into $\mathbb{R}^2$ for which the point $\langle 0,0,\ldots \rangle \in X$ is accessible.
\end{proof}

\bibliographystyle{amsplain}
\bibliography{Bridging}

\providecommand{\bysame}{\leavevmode\hbox to3em{\hrulefill}\thinspace}
\providecommand{\MR}{\relax\ifhmode\unskip\space\fi MR }
\providecommand{\MRhref}[2]{%
  \href{http://www.ams.org/mathscinet-getitem?mr=#1}{#2}
}
\providecommand{\href}[2]{#2}
\begin{thebibliography}{10}

\bibitem{ammerlaan-anusic-hoehn2023}
Andrea Ammerlaan, Ana Anu\v{s}i\'{c}, and Logan~C. Hoehn, \emph{Radial
  departures and plane embeddings of arc-like continua}, Houston J. Math. (to
  appear).

\bibitem{anusic2021}
Ana Anu\v{s}i\'{c}, \emph{Planar embeddings of {M}inc's continuum and
  generalizations}, Topology Appl. \textbf{292} (2021), Paper No. 107519, 13.
  \MR{4223408}

\bibitem{brown1960}
Morton Brown, \emph{Some applications of an approximation theorem for inverse
  limits}, Proc. Amer. Math. Soc. \textbf{11} (1960), 478--483. \MR{115157}

\bibitem{debski-tymchatyn1993}
W.~D\polhk{e}bski and E.~D. Tymchatyn, \emph{A note on accessible composants in
  {K}naster continua}, Houston J. Math. \textbf{19} (1993), no.~3, 435--442.
  \MR{1242430}

\bibitem{problems2018}
Logan~C. Hoehn, Piotr Minc, and Murat Tuncali, \emph{Problems in continuum
  theory in memory of {S}am {B}. {N}adler, {J}r.}, Topology Proc. \textbf{52}
  (2018), 281--308, With contributions. \MR{3773586}

\bibitem{problems2002}
Alejandro Illanes, Sergio Mac\'{i}as, and Wayne Lewis (eds.), \emph{Continuum
  theory}, Lecture Notes in Pure and Applied Mathematics, vol. 230, Marcel
  Dekker, Inc., New York, 2002. \MR{2001430}

\bibitem{lewis1983}
Wayne Lewis, \emph{Continuum theory problems}, vol.~8, 1983, pp.~361--394.
  \MR{765091}

\bibitem{mayer1982}
John~C. Mayer, \emph{Embeddings of plane continua and the fixed point
  property}, ProQuest LLC, Ann Arbor, MI, 1982, Thesis (Ph.D.)--University of
  Florida. \MR{2632164}

\bibitem{mayer1983}
\bysame, \emph{Inequivalent embeddings and prime ends}, vol.~8, 1983,
  pp.~99--159. \MR{738473}

\bibitem{nadler1972}
Sam~B. Nadler, Jr., \emph{Some results and problems about embedding certain
  compactifications}, Proceedings of the {U}niversity of {O}klahoma {T}opology
  {C}onference {D}edicated to {R}obert {L}ee {M}oore ({N}orman, {O}kla., 1972),
  Univ. of Oklahoma, Norman, Okla., 1972, pp.~222--233. \MR{0362256}

\bibitem{nadler-quinn1972}
Sam~B. Nadler, Jr. and J.~Quinn, \emph{Embeddability and structure properties
  of real curves}, Memoirs of the American Mathematical Society, No. 125,
  American Mathematical Society, Providence, R.I., 1972. \MR{0353278}

\bibitem{rogers1970}
J.~W. Rogers, Jr., \emph{On mapping indecomposable continua onto certain
  chainable indecomposable continua}, Proc. Amer. Math. Soc. \textbf{25}
  (1970), 449--456. \MR{256361}

\end{thebibliography}

\end{document}